\documentclass[12pt]{amsart}
\usepackage{amsmath}
\usepackage{amssymb}
\usepackage{amsthm}
\usepackage{dsfont}

\usepackage{graphicx}
\usepackage{color}
\newtheorem{thm}{Theorem}[section]
\newtheorem{cor}[thm]{Corollary}
\newtheorem{lem}[thm]{Lemma}
\newtheorem{prop}[thm]{Proposition}

\theoremstyle{definition}

\theoremstyle{remark}

\numberwithin{equation}{section}

\begin{document}
\title[Fourier transform and regularity of characteristic functions]
{Fourier transform and regularity  of characteristic functions}

\author{Hyerim Ko}
\author{Sanghyuk Lee}

\address{Department of Mathematical Sciences, Seoul National University, Seoul 151--747, Republic of Korea}
\email{shklee@snu.ac.kr} \email{kohr@snu.ac.kr}

\thanks{The authors were supported in part by NRF grant No.2009-0083521 (Republic of Korea).}



\begin{abstract} Let $E$ be a bounded domain in $\mathbb R^d$.
We study regularity property of $\chi_E$ and integrability of
$\widehat {\chi_E }$ when its boundary $\partial E$ satisfies some
conditions. At the critical case these properties are generally
known to fail. By making use of Lorentz and Lorentz-Sobolev spaces
we obtain the endpoint cases of the previous known results. Our
results are based on a refined version of Littlewood-Paley
inequality, which makes it possible to exploit cancellation
effectively.

\end{abstract}

\maketitle

\section{Introduction}
Let $E$ be a bounded measurable subset  in $ \mathbb R^d $. Sobolev
regularity property of $\chi_E$ and $L^p$-integrability of $\widehat
{\chi_E }$ have long been of interest in connection with various
problems and studied extensively by many authors \cite{Herz, Sickel, Triebel2}.
The sharp range of regularity and integrability are relatively well
known and these properties generally fail at the critical exponent.
In extended functions spaces validity of such properties at the
critical exponent is not clearly understood . In this short note
investigate this issue with Lorentz and Lorentz-Sobolev spaces.

\subsubsection*{Integrability of $\widehat {\chi_E}$}
If $2\le p\le \infty$, by the Hausdorff-Young inequality it follows
that $\widehat{\chi_E}\in L^p$. This holds without any dependence on
the geometric structure of $ E $. However, for $ p < 2 $, it becomes
no longer trivial to determine $ L^p$-boundedness of $ \widehat
{\chi_E} $ and the geometric information of $E$ comes into
play, especially the geometry of the boundary of $E$.

In order to describe  the boundary,  we  set
\[\big(
\partial E \big)_\delta=\{x: \text{ dist }(x,\partial E)<\delta\}\]
and  consider the condition that, for $0<\gamma\le d$,
\begin{equation}\label{gamma}
|\big( \partial E \big)_\delta | \lesssim \delta^{d-\gamma}.
\end{equation}
This  is satisfied if $\partial E$ is a $\gamma$ set (see
\cite[Definition 3.1]{Triebel} for definition of $\gamma$-set). As
remarked in (p.5-6 in \cite{Triebel}), the Minkowski content is
equivalent to the Hausdorff measure for $\gamma$-set.

$L^p$-integrability of $\widehat {\chi_E}$ and  Sobolev regularity
of $\chi_E$ are closely related. In fact,  the first on some range
can be deduced from the latter.
We denote by $L_s^q(\mathbb R^d)$, for $0<s<\infty$, the Bessel potential spaces
consisting of all tempered distribution $f\in \mathcal S'(\mathbb
R^d)$ with the norm
$$
\| f \|_{L_s^q} = \Big\| \Big( \Big( 1+ |\xi|^2 \Big)^{\frac
s 2} \widehat{f} (\xi) \Big)^\vee  \Big\|_q,
$$
where $ ^{\vee} $ denotes the inverse Fourier transform
and by $\Lambda_s^{q,r}(\mathbb R^d)$
the Nikol'skij-Besov spaces endowed with the norm
$$
\| f \|_{\Lambda_s^{q,r}} = \| f \|_q  +  \Big( \int_{\mathbb R^d}
\frac{ \left\Vert f(x+t) - f(x) \right\Vert_q^r }{ |t|^{d+rs} } \,
 dt \Big)^{1/r}
$$
for $ 0<s<1 $ and  $q\in [1,\infty]$ and $r\in [1,\infty)$. Then we
see that $\Lambda_s^{q,q} = W_s^q $ where $W_s^q $ denotes the
fractional Sobolev spaces.
Then we have the following characterization of $ W_s^q$ due to Sickel
\cite[Proposition 3.6]{Sickel}. The converse direction is also true
if $E$ is a quasiball \cite[Theorem 1.3]{Faraco}.

\begin{thm}{} \label{sickel}  Let $ 0<s<1$ and $ E $ be a measurable  subset of $\mathbb R^d $
with $|E|<\infty$. Suppose
$$
\int_0^1  \delta^{-qs}\big|\big(  \partial E \big)_\delta \big|
\frac { d\delta}{\delta} <\infty,
$$
then $\chi_E \in W_s^q(\mathbb R^d)$ for $ 1 \le q < \infty $.
\end{thm}

If $\frac{2(d-\alpha)}{d}<p\le2$,  by H\"older's inequality and
Plancherel's theorem  we have
\[\|\widehat {\chi_{E}}\|_p\lesssim
 \|\widehat {\chi_{E}}(1+|\cdot|^2)^\frac{\alpha}{2p}\|_2
\sim  \|\chi_E\|_{W_{{\alpha}/p}^2(\mathbb R^d)}.\] We here use the
 fact that $W_s^2(\mathbb R^d)$ and $L_s^2(\mathbb R^d)$ are
equivalent for $0<s<1$ (see \cite{Stein}). Now, by Theorem
\ref{sickel}, $ \chi_E \in W_{{\alpha}/p}^2(\mathbb R^d) $ if $ 0 <
\frac {\alpha}p < \frac{d-\gamma}2$. Hence,
 $ \widehat { \chi_E} \in L^p(\mathbb R^d) $ for $ p > \frac {2d} {2d-\gamma} $  whenever
 the condition \eqref{gamma} holds and $\gamma \ge d-2$ (see \cite{Lebedev}
 for a slightly different argument).

  Especially for a bounded domain
with $ C^1$-boundary, this (with $\gamma = d-1$ in \eqref{gamma})
recovers the classical result due to Herz \cite{Herz} who showed
that $ \widehat { \chi_E} \in L^p(\mathbb R^d) $ for $ p > \frac
{2d} {d+1} $ under the assumption that $ E $ is compact and convex
with  smoothness condition.

\smallskip

If we assume that $E$ is a bounded domain, an improved
characterization is possible in terms of $\big(  \partial E
\big)_\delta $. The following is our first result.

\begin{prop}\label{fchar} Let $1\le p\le 2$ and $E$ be a bounded domain. Then,
\[\big\| \widehat{\chi_E} \big\|_p  \lesssim |E| + \Big(
\int_0^1 \delta^{-d(1-\frac p 2)} |(\partial E)_{\delta} |^{\frac
p2} \,\frac{d\delta}\delta \Big)^{1/p}.\]
\end{prop}

However, at the critical $ p = \frac{2d}{2d-\gamma}$, the
condition \eqref{gamma} doesn't generally imply $\widehat{\chi_E}\in
L^p(\mathbb R^d)$. For example, if $E$ is a ball  $B$, \eqref{gamma}
is satisfied with $\gamma=d-1$, but $\widehat{\chi_B}\not\in
L^\frac{2d}{d+1}$ because, for $|\xi| \gg 1$,
\begin{equation}\label{bessel} \left| \widehat{ \chi_B} (\xi) \right| \geq
C |\xi|^{-\frac{d+1} 2} |\sin(2\pi |\xi|)| -C'|\xi|^{-\frac{d+3}2}.
\end{equation}
Lebedev \cite{Lebedev} showed that $\widehat{\chi_E}\notin
L^p(\mathbb R^d)$ for $ p \leq \frac {2d}{d+1} $ if  $E$ has $C^2$-
boundary.

Though $\widehat{\chi_E}$ generally fails to be in
$L^\frac{2d}{2d-\gamma}(\mathbb R^d)$ under the assumption
\eqref{gamma}, from the above example it seems natural to expect
that $\widehat{\chi_E}$ is contained in the weaker
$L^{\frac{2d}{2d-\gamma},\infty}$. Here $L^{p,\infty}$ denotes the
weak $L^p$ space. At the critical $ p = \frac {2d} {2d-\gamma}$, we
have the following.

\begin{thm}\label{maintheorem1} Let $ E $ be a bounded domain in $ \mathbb R^d $.
Assume that, for some $ 0< \gamma < d $, \eqref{gamma} holds
for $ 0< \delta \ll 1$. Then $\widehat{ \chi_E }\in
L^{\frac{2d}{2d-\gamma},\infty}(\mathbb R^d)$.
\end{thm}

When $\gamma=d-1$, Theorem \ref{maintheorem1} is optimal by
\eqref{bessel} in that $\widehat{\chi_B}\notin L^{p,\infty}(\mathbb
R^d)$ for $ p <\frac {2d}{d+1}$. For $\gamma$ other than  $d-1$ the
same seems to be true but we are not able to construct an example at
this moment. In particular, if $ E $ has Lipschitz boundary, then $
| \left(
\partial E \right) _\delta | \lesssim \delta $. Hence we get the following.
\begin{cor} \label{C1weak}
Let $ E $ be a bounded domain in $ \mathbb R^d $ with Lipschitz
boundary.  Then $ \widehat{ \chi_E }\in L^{p,\infty} (\mathbb R^d) $
for $p=\frac{2d}{d+1} $.
\end{cor}

\subsubsection*{Regularity property of $\chi_E$}
By Theorem \ref{sickel} $ \chi_E \in W_{s}^q (\mathbb R^d)$
if $s< \frac{d-\gamma}q$ but  $\chi_E$ generally fails to be in
$W_{s}^q (\mathbb R^d)$ at the critical exponent $s =(d-\gamma)/ q$.
In \cite[Theorem 3]{Triebel2} it is shown that,  for $ 1 \le d-1 \le
\gamma < d $ and for $ 1\leq q <\infty$, there is a bounded
star-like domain $E$ such that the boundary $\partial E$ is a
$\gamma$-set and $\chi_E\not\in W_{(d-\gamma)/q}^q (\mathbb R^d)
$.
Another result is that if $E$ is a $K$-quasiball such as Koch snowflake whose boundary has
nonzero $\gamma$-dimensional lower Minkowski content,
 then $\chi_E \notin W_{(d-\gamma)/q}^q
(\mathbb R^d)$ for $1 \le q < \infty$  (Theorem 1.3 in
\cite{Faraco}). This can be shown by observing that if the lower
Minkowski content is nonzero, then the opposite direction of
\eqref{gamma} holds.

%
We can also characterize the regularity of $\chi_E$ by using
the Bessel potential spaces $L_s^q$.
When $q \neq 2$, $L_s^q(\mathbb R^d)$ and $W_s^q(\mathbb R^d)$ do
not coincide in general. But, for $ 1< q <\infty$ and $0<s<1$, there
are the well-known embeddings
\begin{align}
\label{super}
 \Lambda_s^{q,q} \subset L_s^q\,\,  \text{ for  } q\le 2 ,   \quad \Lambda_s^{q,2} \subset  L_s^q\,\, \text{ for } q\ge
 2, \\
 \label{sub}
 L_s^q \subset \Lambda_s^{q,2} \,\, \text{ for  } q\le 2 ,   \quad  L_s^q  \subset  \Lambda_s^{q,q}\,\, \text{ for } q\ge
 2.
 \end{align}
%
(see \cite{Stein}). As a consequence of embedding \eqref{super}, by
the analogous argument in the proof of Theorem \ref{sickel} (see
\cite{Sickel}, also see \cite{Faraco}) it is easy to see that,  for $ q\ge2$,
\begin{align*}\label{lsq}
\| \chi_E \|_{L_s^q} \lesssim |E|^{1/q}+ \Big(
\int_0^1 \delta^{-2s} | (\partial E)_\delta | ^{\frac2q }
\frac{\mathrm d\delta} \delta \Big)^{1/2},
\end{align*}
and, for $q\le 2$,
\begin{align*}
\| \chi_E \|_{L_s^q} \lesssim |E|^{1/q}+ \Big(
\int_0^1 \delta^{-qs} | (\partial E)_\delta |\frac{\mathrm d\delta}
\delta \Big)^{1/q}.
\end{align*}
Even though these inequalities  give  better control of $\|
\chi_E \|_{L_s^q}$ in terms of $|(\partial E)_\delta|$,
these do not give any information at the critical exponent $ s
=(d-\gamma)/ q$. Furthermore, the aforementioned examples of domains
$E$ for which $\chi_E\not\in W_{(d-\gamma)/q}^q (\mathbb R^d) $ and
the embedding \eqref{sub} for $q\ge 2$ (also using $\Lambda_s^{q,q}
= W_s^q $) show  that $\chi_E$ is not generally contained in $L_s^q$
at the critical exponent $ s =(d-\gamma)/ q$. This leads us to
consider the Lorentz-Sobolev space.

For a tempered distribution $ f \in \mathcal S' (\mathbb R^d)$, we
set $ f_s $ by $ \widehat{f_s}(\xi) =\left(  1 + |\xi|^2
\right)^{\frac s 2 } \widehat{f}(\xi)$. Then the Lorentz-Sobolev
spaces $ L_s^{q,r} (\mathbb R^d) $ are defined by the norm
$
  \| f \|_{L_s^{q,r}}
= \| f_s \|_{L^{q,r}}
$
for $1\le q <\infty$ and $1\le r \le \infty $ where $L^{q,r}$ denote the Lorentz spaces.
By the Chebyshev inequality, the Bessel potential
 spaces $L_s^q (\mathbb R^d) $ are embedded in
 $ L_s^{q,\infty}(\mathbb R^d) $.
For more detail on Lorentz-Sobolev spaces we refer the reader to the
recent literature \cite{Hajaiej, Machihara, Wadade, Xiao}.

 In what follows we show  that $\chi_E\in L_s^{q,\infty}(\mathbb R^d)$ at the critical
  $ s= (d-\gamma)/q
$.

\begin{thm}\label{maintheorem2}
Let $E$ be a bounded domain in $\mathbb R^d$ satisfies \eqref{gamma}
for $ 0<\delta \ll 1$. If $0<\gamma<d$ and $1<q<\infty$, then $
\chi_E \in L_s^{q,\infty}(\mathbb R^d) $ for $ s=(d-\gamma)/q$.
\end{thm}

The paper is organized as follows: In Section 2, we prove a refined
Littlewood-Paley inequality in  which the projection operators have
preferable cancellation property. Theorem \ref{maintheorem1} and
Proposition \ref{fchar} are proved in Section 3. The proof of
Theorem \ref{maintheorem2} is given in Section 4.

\section{ Preliminaries }

In this section we prove a version of Littlewood-Paley inequality 
which plays a crucial role for the proof of our results. Most
important feature is that the associated projection operators have a
cancellation property when they are applied to the characteristic
functions of open sets. For this purpose we need to find a smooth
function $ \phi \in \mathcal S(\mathbb R^d)$ which satisfies special
properties.

We denote by $B_r(a)$ the open ball of radius $r$ which is centered
at the point $a$.

\begin{lem} \label{keylemma}
There exists a Schwartz function $\phi$ such that  $ \phi ^\vee $ is
supported on $B_1(0)$ and $\phi$ satisfies
    \begin{equation}\label{lem2}
    \int_{\mathbb R^d} \phi ^\vee (x) \,\mathrm dx=0,
    \end{equation}
and, for some constants $ C_1 $, $C_2 >0$,
    \begin{equation}\label{lem1}
    C_1 \leq \sum_{ k=-\infty}^{\infty} \phi^{2}( 2^{ -k }\xi) \leq
    C_2.
    \end{equation}
Moreover, for any positive integer $N$,
    \begin{equation}\label{lem3}
    \int_{\mathbb R^d} x^\beta \phi ^\vee (x) \,\mathrm dx=0 \quad
    \text{if} \quad |\beta|< 2^N.
    \end{equation}
\end{lem}
\begin{proof}
Choose a radial function $\psi_0\in\mathcal S (\mathbb R^d)$ such
that $ \psi_0 $ is supported on $B_{1/2}(0)$ and
$ \int \psi_0(x) \mathrm d x=\widehat{\psi_0}(0)\neq0. $
Then we select $\phi$ by setting
$$
\phi^\vee (x)=\psi_0(x)-2^{-d}\psi_0(2^{-1}x)
$$
and after a change of variables \eqref{lem2} follows.

We now prove the estimate \eqref{lem1}. For scaling it suffices to
prove \eqref{lem1} for $ 1 \leq |\xi| \leq 2 $. Since $\phi^2(0) =
0$, we have $|\phi^2(2^{-k} \xi)| \leq C 2^{-k}|\xi|$. Hence,
    \[
    \begin{aligned}
    \sum_{ k = -\infty}^{\infty}\phi^2( 2^{ -k }\xi)
    \lesssim& \sum_{ k =-\infty}^{\infty}\text{min} \left( 2 ^ { - k }|\xi|, (2 ^ { -k }| \xi |)^{-1} \right)
    \leq C_2.
    \end{aligned}
    \]
This gives the upper bound of \eqref{lem1}. For the lower bound note
that $\phi(\xi)=\widehat{\psi_0}(\xi)-\widehat{\psi_0}(2\xi)$. Since
$\sum_{k=-\infty}^{\infty}\phi(2^{-k}\xi)=\widehat{\psi_0}(0)\neq0$
converges uniformly for $ 1 \leq |\xi|\le  2$, there exists $ i_0 \in
\mathbb Z_+ $ such that
   $
     \frac{1}{2}\big|\widehat{\psi}(0)\big|
    \leq
    \big|\sum_{ k = - i_0 }^{ i_0 }\phi(2^{-k}\xi) \big|.
    $
By the Cauchy-Schwarz inequality, we see
    \begin{equation}\label{cauchy}
    \frac{1}{2}|\widehat{\psi}(0)|
    \leq  ( 2i_0+1 )^{\frac{1}{2}}
    \Big(\sum_{k = -i_0}^{ i_0
    }\phi^2(2^{-k}\xi)\Big)^\frac{1}{2}\le
    ( 2i_0  +1 )^{\frac{1}{2}}
    \Big(\sum_{k = -\infty}^{ \infty
    }\phi^2(2^{-k}\xi)\Big)^\frac{1}{2},
    \end{equation}
which gives the desired uniform lower bound of \eqref{lem1} for $ 1
\leq |\xi|\le  2$.

We now have \eqref{lem3} with $\beta=0$.
We may assume that $\phi^\vee (x)$ is supported
in $B_{2^{-N}}(0)$ by replacing $\phi$ with
$\phi(2^{-N}\cdot)$. In order to have \eqref{lem3} with bigger
$|\beta|<2^N $ we need only to consider $\phi^{2^N}(\xi)$ instead of $
\phi (\xi)$. Then $(\phi^{2^N})^\vee (x)$ is supported on $B_1$ and
$\partial^\beta (\phi^{2^N})(0)=0$ for $|\beta|< 2^N$ which gives
\eqref{lem3}. As before, the estimate \eqref{lem1} follows by
applying the Cauchy-Schwarz inequality $N+1$ times.
\end{proof}

We now  prove the Littlewood-Paley inequality in which the
projection operators are defined by the Schwartz function in Lemma
\ref{keylemma}, and  give a characterization of the Lorentz-Sobolev
spaces. The associated projection operators $P_{\le0}$ and
$P_k$ for $ k \geq 1 $ are similar to the classical
Littlewood-Paley operators but they are different in that the
multiplier does not have compact support. However, the standard
argument works without much of modification. For the reader's
convenience we include a proof.

\medskip

Let $\phi$ be given as in Lemma \ref{keylemma} and define
\begin{align}\label{pkf}
\widehat{P_k f}(\xi) = \phi^{2} \left( 2^{-k} \xi \right)
\widehat{f}(\xi), \quad \widehat{ P_{\le0} f} (\xi) = \Phi_{0}( \xi)
\widehat f (\xi)
\end{align}
where $ \Phi_0 (\xi) = (1+|\xi|^2)^{s/2}\sum_{k=-\infty}^{0} \phi^4 \left( 2^{-k}
\xi \right). $ In what follows we prove  Littlewood-Paley inequality
in Lorentz-Sobolev spaces.

\begin{lem}\label{wood}
Let $ s>0 $ and $ 1< q <\infty$, $1\le r\le \infty$. Then there
exists a constant $C=C(d,q,s)$ such that, for $ f \in L_s^{q,r}
(\mathbb R^d) $,
\begin{equation}\label{Pkleft}
 \|  P_{\le0} f \|_{q,r}
   + \Big\|  \Big(  \sum_{ k =
1 }^\infty  \Big( 2^{ks}|P_k f| \Big) ^2
 \Big) ^{\frac 1 2} \Big\|_{q,r}
\leq C \| f \|_{L_s^{q,r}}
\end{equation}
and, for $ f \in \mathcal S' (\mathbb R^d) $,
\begin{equation}\label{Pkright}
C^{-1} \| f \|_{L_s^{q,r}} \leq   \| P_{\le0} f \|_{q,r}
+ \Big\| \Big(  \sum_{ k = 1 }^\infty \Big( 2^{ks}|P_k f|
\Big) ^2
 \Big) ^{\frac 1 2}  \Big\|_{q,r} .
\end{equation}

\end{lem}
Before proving this lemma, we need to prove the following estimates
which allow us to use  the Mikhlin multiplier theorem.

\begin{lem}\label{cond}  Let $s>0$ and $ \phi $ be given as  in Lemma
\ref{keylemma}.  Also, let $ m_k $ and $ m_k'$ be given
by
\begin{equation}\label{multiplier}
m_{k} ( \xi ) =    2^{ks}
 \phi^{2} ( 2^{-k}\xi)  (1+|\xi|^2)^{-\frac s 2},
\end{equation}
\begin{equation}\label{multiplier2}
m_{k}' ( \xi ) =    2^{-ks}
 \phi^{2} ( 2^{-k}\xi)  (1+|\xi|^2)^{\frac s 2} \Big( \sum_{j=-\infty}
^\infty \phi^4 (2^{-j}\xi) \Big)^{-1}.
\end{equation}
Let $ \{ \omega_k \} $ be a sequence of constants having values $
\pm 1 $. Then
\begin{equation}\label{mikh}
\Big|  \partial_\xi ^ \alpha \Big( \sum_{ k=1 }^L\omega_k m_{k} (
\xi ) \Big)  \Big| \lesssim  | \xi | ^ {- | \alpha | }
\end{equation}
for all $ |\alpha| \leq \frac d 2 +1 $ and an analogue of
\eqref{mikh}  also holds if  $m_k$  is replaced with
$m_k'$.
\end{lem}

Here we remark that $\sum_{j=-\infty} ^\infty \phi^4 (2^{-j}\xi)\sim
1$ so that  $\big( \sum_{j=-\infty} ^\infty \phi^4 (2^{-j}\xi)
\big)^{-1}$ is well defined. In fact, this can be shown by applying the
Cauchy-Schwarz inequality to \eqref{cauchy}.

\begin{proof}
Let $\mu$ be a multi-index.  If we choose a sufficiently large $N$
in Lemma  \ref{keylemma}, then \eqref{lem3} guarantees that $
 \partial^\mu ( \partial^{\beta} \phi) (0) = 0$ for $ |\beta| \leq
\frac d 2 +1 $ and $|\mu|\le N$. Hence, for $ |\beta|
\leq \frac d 2 +1 $ and $M>0$,
\begin{equation}\label{betabd}
\big| \partial ^{\beta} (\phi^2 (2^{-k}\xi)) \big| \lesssim 2^{-k|\beta|} \min
\Big\{ \left( 2^{-k} |\xi| \right)^{N} , (2^{-k} |\xi|) ^
{-M} \Big\}.
\end{equation}
Using  \eqref{betabd} with large enough $N, M$,  we see that, for $
|\alpha| \leq \frac d 2 +1, $
\begin{align*}
&\Big|  \partial^ { \alpha } \Big( \sum_{ k=1 }^L \omega_k  m_{k}
( \xi ) \Big)  \Big| \lesssim \sum_{k=1}^L \sum_{\beta+\gamma =
\alpha} 2^{ks} \Big|\partial^{\beta} \big\{ \phi^2(2^{-k}\xi)\big\}
\partial^{\gamma} \big\{ (1+|\xi|^2)^{-\frac s 2} \big\} \Big|
\\
&\lesssim  \sum_{k=1}^\infty \sum_{\beta+\gamma=\alpha}
2^{k(s-|\beta|)} \min \big\{ \left( 2^{-k} |\xi| \right)^{N}
 , (2^{-k} |\xi|) ^ {-M} \big\} |\xi|^{-s-|\gamma|} \lesssim |
\xi |^ { -|\alpha| }.
\end{align*}
This gives the desired inequality \eqref{mikh}.
Similarly, $|\partial^ { \alpha }(\sum_{ k=1 }^L
\omega_k m_k' ( \xi ))|$ is bounded by
\begin{align*}
&\sum_{k=1}^L\sum_{\beta+\gamma+\delta= \alpha} 2^{-ks}
\Big|\partial^{\beta}  \big\{ \phi^2(2^{-k}\xi)\big\}\Big|
 \Big|\partial^{\gamma}\big\{ (1+|\xi|^2)^{\frac s 2} \big\} \Big|\Big|
\partial^{\delta} \Big\{ \Big(\sum_{j=-\infty}^\infty \phi^4
(2^{-j}\xi) \Big)^{-1} \Big\}\Big| \\
%
%
&\lesssim \sum_{ k=1 }^\infty
\sum_{\beta+\gamma+\delta=\alpha} 2^{-k(s+|\beta|)} \min \left\{
\left( 2^{-k} |\xi| \right)^{N} , (2^{-k} |\xi|) ^ {-M}
\right\} (1+|\xi|^2)^{s/2}|\xi|^{-|\gamma|-|\delta|}
\end{align*}
by \eqref{betabd} with sufficiently large $N,M$
and is bounded by $ |\xi|^{-|\alpha|} $. 
This completes the proof.
\end{proof}

Now we prove Lemma \ref{wood}.
\begin{proof}
To prove the first part of Lemma \ref{wood}, we make use of the
Mikhlin multiplier theorem (cf. \cite[Theorem 5.2.7]{Grafakos}) and
Khintchine's inequality. By the standard density argument we may
assume  that  $f$ is contained in the Schwartz class.

As for the operator $ P_{\le0} $, by following the argument in the
proof of Lemma  \ref{cond} it easy to see that the multiplier $
\sum_{k=-\infty}^0 \phi^4(2^{-k}\xi) $ satisfies $|\partial^\alpha
(\sum_{k=-\infty}^0 \phi^4(2^{-k}\xi))|\lesssim |\xi|^{-|\alpha|}$.
Hence, by the Mikhlin multiplier theorem  we have $ \| P_{\le0}f \|_q
\lesssim \|f_s\|_q$ for $1<q<\infty$. Then the  Lorentz bound
\begin{equation}\label{loren}
\| P_{\le0}f \|_{q,r} \lesssim \|f_s\|_{q,r} \end{equation} follows
from the Marcinkiewicz interpolation theorem between the $L^q$ estimates.

To  bound the square function of \eqref{Pkleft}, we consider the
multiplier $ m_k(\xi) $ defined by \eqref{multiplier}. Let $ \{
\omega_k \} $ be independent random variables taking values $ \pm 1
$ with equal probability. Since $ ( m_k \widehat{f_s} )^{\vee} =
2^{ks} P_kf $,  Khintchine's inequality gives, for $ 0 < q < \infty
$,
\begin{equation}\label{zero}
  \Big( \sum_{ k = 1}^L \Big( 2^{ks} |P_k f|\Big)^2
 \Big)^{\frac q 2}
\approx \mathbb E \Big( \Big|
 \sum_{ k=1 } ^ L
\Big( \omega_k   m_{k} \widehat {f_s} \Big)^\vee
    \Big|^q \Big)
\end{equation}
with the implicit constants independent of  $ L $. Thanks to
\eqref{mikh} in Lemma \ref{cond} we can apply the Mikhlin multiplier
theorem to the right-hand side of \eqref{zero}. Taking integral both
side of \eqref{zero},  passing to the limit $ L \to \infty $ and
using Fatou's lemma  give
\begin{equation}\label{pk_pp}
\Big\|  \Big( \sum_{ k = 1}^\infty \Big(2^{ks}|P_k f|\Big)^2
 \Big)^{\frac 1 2}  \Big\|_ q^q
\lesssim \|  f_s \|_q^q
\end{equation}
for all $ 1 < q < \infty $. Then \eqref{Pkleft} is an immediate
consequence of the Marcinkiewicz interpolation theorem and this
proves the first part \eqref{Pkleft} of Lemma \ref{wood}.

Now we show the inequality \eqref{Pkright} by using the duality
argument. Let $ f, g $  be Schwartz functions. By the Plancherel
identity we have \[ \int f_s(x) \overline{ g }(x) \, \mathrm dx=
\int \widehat{f_s}(\xi) \overline{\widehat g}(\xi)  \, \mathrm
d\xi.\]
Using the identity $ 1 = \left( \sum_{k=-\infty}^0 \phi^4(2^{-k}\xi)
+ \sum_{k=1}^\infty \phi^4(2^{-k}\xi) \right)
\big(\sum_{j=-\infty}^\infty \phi^4(2^{-j}\xi) \big)^{-1} $, we
decompose  $\widehat{f_s} \overline{\widehat g}$  so that
\begin{align*}
\widehat { f_s }  \overline { \widehat{  g } } & =\widehat{P_{\le0}
f} m_0' \overline{\widehat{g}}+ \sum_{ k=1 } ^ {\infty}  2^{ks}
\widehat { P_k f } m_k' \overline{\widehat{g}   },
\end{align*}
where $ m_0' (\xi) = \big( \sum_{j=-\infty}^\infty
\phi^4(2^{-j}\cdot) \big)^{-1}$ and
 $ m_k' $ is defined in \eqref{multiplier2}.
By repeated use of the Plancherel identity, we have
\begin{equation}\label{dual}
\int
f_s(x) \overline{ g }(x)
\, \mathrm dx
=  \int P_{\le0} f(x) \overline {\big( m_0' \widehat{g}
\big)^{\vee}}(x)
 \,\mathrm dx  + \sum_{ k=1 } ^ {\infty} \int 2^{ks}
 P_k f ( x ) \overline{ \big(
m_k' \widehat{g} \big)^\vee}
 ( x ) \,  \mathrm dx .
\end{equation}

Let $ 1 < q < \infty $ and $ 1\le r \le\infty$
satisfying $ \frac 1 q + \frac 1 {q'} = 1 $ and $ \frac 1r +\frac 1{r'}=1$.
We may apply the Cauchy-Schwarz inequality and the H\"older-type inequality
for Lorentz spaces to obtain
\begin{align*}
&\Big| \int f_s(x) \overline{ g }(x) \, \mathrm dx  \Big| \leq  \|
P_{\le0} f \|_{q,r} \| \big( m_0' \widehat{g} \big)^{\vee}
\|_{q',r'} \\
& \qquad\qquad +  \Big\| \Big( \sum_{ k=1 } ^ {\infty} \Big( 2^{ks}
 |P_k f|  \Big) ^2 \Big) ^ \frac 1 2     \Big\| _{q,r}
\Big\| \Big( \sum_{ k=1 } ^ {\infty} |
\big( m_k' \widehat{g} \big)^\vee  | ^2 \Big) ^ \frac 1 2  \Big\|_{q',r'}.
\end{align*}
Then it is easy to see that $ \|\big( m_0' \widehat{g}
\big)^{\vee} \|_{q',r'} $ is bounded by $ \| g \|_{q',r'} $ by
following the same argument which shows the boundedness of
$P_{\le0}$ (see \eqref{loren}). Since $m_k'$ also satisfies \eqref{mikh} in the place of
$m_k$, by repeating the argument for \eqref{Pkleft} it follows that
\[ \Big\| \Big( \sum_{ k=1 } ^ {\infty} |
\big( m_k' \widehat{g} \big)^\vee  | ^2 \Big) ^ \frac 1 2
\Big\|_{q',r'} \lesssim  \| g \|_{q',r'} .  \] Hence, combining
these two estimates  we have
\begin{align*}
\Big| \int f_s(x) \overline{ g }(x) \, \mathrm dx  \Big|
 \lesssim \| P_{\le0} f \|_{q,r} \| g \|_{q',r'}
+ \Big\| \Big( \sum_{ k=1 } ^ {\infty} \Big( 2^{ks}
 |P_k f|  \Big) ^2 \Big) ^ \frac 1 2     \Big\|_{q,r}\| g \|_{q',r'}
 .
\end{align*}
Finally, taking supremum over Schwartz function $g$ with  $ \| g
\|_{q',r'} \leq 1$ gives the desired inequality \eqref{Pkright}.
\end{proof}

By making use of the function $\phi$ in Lemma \ref{keylemma} we
prove the following which relates the $L^p$-norm of $P_k\chi_E$ and
the measure of $(\partial E)_{2^{-k}}$.

\begin{lem}\label{packet}
Let $E$ be a bounded domain in $ \mathbb R^d $ and $ 1\leq p < \infty$.
If a Schwartz function $ \phi^\vee$
is supported on $ B_1(0) $ and
satisfies \eqref{lem2}, then
\begin{equation}\label{nbd}
\Big\| \Big( \phi(2^{-k}\xi)\widehat{\chi_E}(\xi) \Big)^{\vee} \Big\|_p
\lesssim | (\partial E)_{2^{-k}} |^{1/p}.
\end{equation}
Additionally, if we use $\phi^2$ instead of $\phi$,
\begin{equation}\label{PK}
\big\| P_k\chi_E \big\|_{p}
=\Big\| \Big( \phi^2(2^{-k}\xi)\widehat{\chi_E}(\xi) \Big)^{\vee} \Big\|_{p}
\lesssim |(\partial E)_{2^{-k+1}}|^{1/p}.
\end{equation}
\end{lem}

\begin{proof}
Fix $ p $,  $ 1 \leq p < \infty $, and consider
\begin{align*}
\big\| \big( \phi(2^{-k}\xi)\widehat{\chi_E}(\xi) \big)^{\vee} \big\|_p
    =&\Big(\int_{\mathbb R^d}\Big|\int_{\mathbb R^d}\phi^\vee(y)
    \chi_E ( \xi-2^{-k}y )\,dy \,\Big|^p\,d\xi\Big)^{1/p}.
\end{align*}
The crucial observation is that for any fixed $ k $,
the inner integral above survives
only for $ \xi $ such that $\text{dist }(\xi,\partial E)
 \leq 2^{-k}$.
This is because $ \phi^\vee $ is supported on $ B_1 $ and then
the integral vanishes for $ \xi \in E $ such that
$\text{dist }(\xi,\partial E) > 2^{-k}$ due to
\eqref{lem2}.
As a consequence, the $L^p$-norm must be smaller than
$ \left| ( \partial E)_{2^{-k}} \right|^{1/p} $
as desired.

For the second inequality we need only to observe that $\int ( \phi^2)^\vee=0$ and  $( \phi^2)^\vee$ is supported in
$B_2(0)$. The same argument gives \eqref{PK}.
\end{proof}

\section{Proof of Theorem \ref{maintheorem1} and
Proposition \ref{fchar}}

In this section we prove Theorem \ref{maintheorem1}
and Proposition \ref{fchar}.

\begin{proof}[Proof of Theorem \ref{maintheorem1}]
Let $ \phi $ be given as
in Lemma \ref{keylemma}. Then,
we have
\begin{equation}\label{dee}
\big|\widehat{ \chi_E }(\xi)\big| \lesssim \sum_{k =
-\infty}^{\infty}\phi^2 ( 2^{ -k }\xi)\big|\widehat{ \chi_E
}(\xi)\big|.
\end{equation}
To get boundedness in the weak $L^p$ spaces, we separately handle
the $ L^p $-norm of $\phi(2^{-k}\cdot)\widehat{\chi_E}$. Using Lemma
\ref{packet} with $ p=2 $ and the condition \eqref{gamma}, we have
$ \big\|\phi(2^{-k}\xi) \widehat{ \chi_E }(\xi) \big\|_2
   $ $\lesssim 2^{-\frac{k (d-\gamma)}{2}}$.
Hence, using the Cauchy-Schwarz inequality, we get
    \begin{equation}\label{l1norm}
    \big\|\phi^2 ( 2^{ -k }\xi) \widehat{ \chi_E }(\xi) \big\|_1
    \lesssim 2^{\frac{k \gamma}{2}}
    \end{equation}
and by \eqref{PK}
    \begin{equation}\label{l2norm}
    \big\|\phi^2 ( 2^{ -k }\xi) { \widehat{\chi_E} }(\xi) \big\|_2
    \lesssim 2^{-\frac{k (d-\gamma)}{2}}.
    \end{equation}

Let $N$ be an integer to be chosen later.
We consider the distribution function of \eqref{dee}
and apply Chebyshev's inequality
so that
    \begin{align*}
   & \,\,\, \,\,\, \,\,\Big| \Big\{  \xi: \sum_{k= -\infty}^{\infty}
    \phi^2 ( 2^{ -k }\xi)\big|\widehat{ \chi_E }(\xi)\big|>\lambda \Big\}\Big| \\
    &\leq
    \Big| \Big\{\xi: \sum_{k= -\infty}^{N}
    \phi^2 ( 2^{ -k }\xi)\big|\widehat{ \chi_E }(\xi)\big| >\frac{\lambda}{2}\Big\}
    \Big|
    +\Big| \Big\{\xi: \sum_{k=N+1}^{\infty}
    \phi^2 ( 2^{ -k }\xi)\big|\widehat{ \chi_E }(\xi)\big| >\frac{\lambda}{2}\Big\}             \Big|\\
    &\lesssim \lambda^{-1} \Big\|\sum_{k= -\infty}^{N}
    \phi^2 ( 2^{ -k }\xi)\big|\widehat{ \chi_E }(\xi)\big|\,\Big\|_1
    +\lambda^{-2} \Big\| \sum_{k=N+1}^{\infty}
    \phi^2 ( 2^{ -k }\xi)\big|\widehat{ \chi_E }(\xi)\big|\,\Big\|_2^{2}.
    \end{align*}
The last term is bounded by
\[
\lambda^{-1}\sum_{k=-\infty}^N \big\|\phi^2 ( 2^{ -k }\xi)
\widehat{ \chi_E }(\xi)\,\big\|_1
+\lambda^{-2} \Big(\sum_{k=N+1}^\infty \big\|\phi^2 ( 2^{ -k }\xi)
\widehat{ \chi_E }(\xi)\,\big\|_2 \Big)^2
\]
by the Minkowski inequality. By application of the estimates \eqref{l1norm} and
\eqref{l2norm} and summation along $k$, we get
\begin{align*}
    \Big| \Big\{  \xi: \sum_{k= -\infty}^{\infty}
        \phi^2 ( 2^{ -k }\xi)\big|\widehat{ \chi_E }(\xi)\big|>\lambda \Big\}\Big|
        \lesssim  \lambda^{-1}
        2^{\frac{N \gamma}{2}}
        + \lambda^{-2}        2^{-N (d-\gamma)}.
\end{align*}
If we take $ N $ such that $ 2^N \approx \lambda^{-2/(2d-\gamma)} $,
the right hand side is bounded by $ \lambda^{-2d/(2d-\gamma)}$. Hence,
$\widehat{ \chi_E }\in
L^{\frac{2d}{2d-\gamma},\infty}(\mathbb R^d)$ and
this concludes the proof.
\end{proof}

In what follows we prove  Proposition \ref{fchar}. This is done by relating the $ L^p $-norm
 of $\widehat{\chi_E} $ to the integral of $|(\partial E)_\delta|$.

\begin{proof}[Proof of Proposition \ref{fchar}]
From \eqref{nbd} with $p=2$, $\big\|\phi ( 2^{ -k }\xi) { \widehat{\chi_E} }(\xi) \big\|_2 \lesssim |(\partial E)_{2^{-k}}|^{1/2}$.
For $ 1\le p \le 2 $, H\"{o}lder's inequality yields
\begin{equation}\label{rieszin}
\big\|\phi^2 ( 2^{ -k }\xi) { \widehat{\chi_E} }(\xi) \big\|_p
\lesssim
2^{kd(\frac 1 p -\frac 1 2)}|(\partial E)_{2^{-k}}|^{1/2}.
\end{equation}
Since $
\sum_{k=-\infty}^\infty \phi^4(2^{-k}\xi)$ is bounded below
as indicated in Lemma \ref{cond}, we get
\begin{align*}
\big\| \widehat{\chi_E} \big\|_p &\lesssim \Big\|
\sum_{k=-\infty}^0 \phi^4(2^{-k}\xi) |\widehat{\chi_E}(\xi)| \Big\|_p
+
\Big\|\sum_{k=1}^\infty \phi^4 ( 2^{ -k }\xi)  |\widehat{\chi_E} (\xi)| \Big\|_p.
\end{align*}
Clearly, the first term in the right-hand side is bounded by
$C\| \widehat{\chi_E} \|_\infty $, which is in turn bounded by $|E|$. For the second term  it follows  by
H\"older's inequality  and  \eqref{rieszin} that
\begin{align*}
\Big\|\sum_{k=1}^\infty \phi^4 ( 2^{ -k }\xi)  \widehat{\chi_E} (\xi) \Big\|_p^p &\le  \Big\|  \Big( \sum_{k=1}^\infty 
|\phi^2(2^{-k}\xi)\chi_E(\xi)|^p \Big)^{1/p}\Big\|_p^p
&\lesssim  \sum_{k=1}^\infty 2^{kd (1-\frac
p 2)}|(\partial E)_{2^{-k}} |^{\frac p2}.
\end{align*}
Note that the last sum is bounded by
$
\int_0^1 \delta^{-d (1 -\frac p 2)}
 |(\partial E)_{\delta} |^{\frac p2}
\frac{\mathrm d\delta}\delta
$
because $|(\partial E)_{\delta} |$ is increasing in $\delta$.
Therefore, combining these estimates gives
\begin{align*}
\big\| \widehat{\chi_E} \big\|_p
&\lesssim |E| + \Big(\int_0^1 \delta^{-d (1 -\frac p 2)}
 |(\partial E)_{\delta} |^{\frac p2}
\frac{\mathrm d\delta}\delta \Big)^{1/p}.
\end{align*}
 This completes the proof.
\end{proof}

\section{Proof of Theorem \ref{maintheorem2}}

In this section we prove Theorem \ref{maintheorem2}
by using the Littlewood-Paley inequality in the Lorentz-Sobolev spaces $L_s^{q,\infty}$  which  we have proved in Lemma \ref{wood}.

\begin{proof}[Proof of Theorem \ref{maintheorem2}]
Let $1<q<\infty,$ and $s=(d-\gamma)/q$. For the  proof of Theorem \ref{maintheorem2} it is sufficient to  show
that the right-hand side of
\eqref{Pkright} is finite while $f=\chi_E$ and $r=\infty$.

To estimate the first term,
we note that the multiplier $\Phi_0$ of the operator $P_{\le 0}$ satisfies
the Mikhlin multiplier condition by an analogous proof of \eqref{loren}.
In particular, we have
\begin{equation}
\label{trivial}
 \|P_{\le0}\chi_E\|_{q,\infty}\lesssim \|\chi_E\|_{q,\infty}\sim |E|^\frac1q.
 \end{equation}

We now examine the square function which appears in \eqref{Pkright}. Let us choose
$ p_0 $ and $ p_1 $ such that
$1 < 2p_1 < q < 2p_0<\infty$, and as before let $N$ be an integer to be chosen later.
By a simple manipulation and Chebyshev's inequality we see
\begin{align*}
&\,\,\,\,\,\,\,\,\,\Big| \Big\{ x :  \Big(  \sum_{ k =1} ^ \infty  \Big(
2^{ks} | P_k  \chi _E (x)  | \Big) ^ 2
    \Big) ^{1/2}    >    \lambda   \Big\} \Big| \\
&\leq \, \Big| \Big\{ x :    \sum_{ k =1} ^ N  
2^{2ks} | P_k \chi _E  (x)  | ^ 2
      >   \frac {\lambda^2} 2  \Big\} \Big|
+ \Big| \Big\{ x :    \sum_{ k = N+1} ^ \infty 
 2^{2ks} | P_k  \chi _E (x)   | ^ 2
      >  \frac{ \lambda^2 } 2  \Big\} \Big|  \\
&\lesssim \,
  \lambda^{-2p_0}\Big\| \sum_{ k = 1} ^ N 
  2^{2ks} | P_k \chi _E   |^ 2
 \Big\| _ {p_0 }^ {p_0}
+
 \lambda^{-2p_1} \Big\| \sum_{ k = N+1} ^ \infty 
 2^{2ks} | P_k \chi _E   | ^ 2
 \Big\| _ {p_1} ^ {p_1}.
\end{align*}
By Minkowski's inequality, the above sum is bounded by
\[
\lambda^{-2p_0} \Big(
\sum_{k=1}^N 2^{2ks} \| (P_k\chi_E)^2 \|_{p_0}
\Big)^{p_0}
+
\lambda^{-2p_1} \Big(
\sum_{k=N+1}^\infty 2^{2ks} \| (P_k\chi_E)^2 \|_{p_1}
\Big)^{p_1}.
\]
Applying \eqref{PK} with $p=2p_0$ and $p=2p_1$ respectively
and \eqref{gamma} with $d-\gamma=sq$, we get
\begin{align*}
\Big| \Big\{ x :  \Big(  \sum_{ k =1} ^ \infty  &\Big(
2^{ks} | P_k  \chi _E (x)  | \Big) ^ 2
    \Big) ^{1/2}    >    \lambda   \Big\} \Big| \\
    &\lesssim \lambda^{-2p_0} \Big(
\sum_{k=1}^N 2^{2ks} 2^{-ksq/{p_0}} \Big)^{p_0}
+\lambda^{-2p_1} \Big( \sum_{k=N+1}^\infty 2^{2ks} 2^{-ksq/{p_1}}
\Big)^{p_1} \\
    &\lesssim \lambda^{-2p_0} 2^{Ns(2p_0-q)} + \lambda^{-2p_1}2^{Ns(2p_1-q)}.
\end{align*}
Choosing  $ N $ to be $2^{Ns} \approx \lambda $, the right hand side is bounded by $C\lambda^{-q}$. Hence,
\[ \Big\| \Big(  \sum_{ k =1} ^ \infty  \Big(
2^{ks} | P_k  \chi _E (x)  | \Big) ^ 2
    \Big) ^{1/2}\Big\|_{q,\infty} \lesssim 1 .\]
Combining this with \eqref{trivial} and using \eqref{Pkright}  we conclude that
$ \| \chi_E \|_{L^{q,\infty}_s(\mathbb R^d)}<\infty $.  This
completes the proof.
\end{proof}

\bibliographystyle{plain}

\end{document}